\documentclass[a4j,12pt,final]{article}
\usepackage{amsmath}
\usepackage{amssymb}
\usepackage{amsthm}
\usepackage{amscd}
\newtheorem{theorem}{Theorem}[section]

\theoremstyle{definition}
\newtheorem{definition}{Definition}[section]
\theoremstyle{remark}

\theoremstyle{proposition}
\newtheorem{proposition}{Proposition}[section]
\numberwithin{equation}{section}

\theoremstyle{corollary}

\theoremstyle{conjecture}

\begin{document}

\title{The Euler class in the Simplicial \\
de Rham Complex}
\author{Naoya Suzuki}
\date{}
\maketitle
\begin{abstract}
 We exhibit a cocycle in the simplicial de Rham complex
which represents the Euler class.
As an application, we construct a Lie algebra cocycle on $L\mathfrak{so}(4)$.
\end{abstract}

\section{Introduction}

For any Lie group $G$, we can define a simplicial manifold $\{ NG(*) \}$ and  a double complex ${\Omega}^{*} (NG(*)) $ on it.
In classical theory, it is well-known that the cohomology ring of the total complex
${\Omega}^{*} (NG) $ is isomorphic to $H^*(BG)$ where $BG$ is a classifying space of $G$, which is not a manifold in general \cite{Bot2} \cite{Dup2} \cite{Mos}.

 In \cite{Dup}, Dupont introduced another double complex $A^{*,*}(NG)$ on $NG$ such that the cohomology ring of its total complex $A^*(NG)$
is also isomorphic to $H^*(BG)$. He used it to construct a homomorphism from $I^* (G)$, the $G$-invariant polynomial ring over Lie algebra $\mathcal{G}$, to $H^*(BG)$.
By using Dupont's method,  in \cite{Suz} the author exhibited cocycles in ${\Omega}^{*} (NG) $ which represent the Chern characters.
In this paper, we will exhibit cocycles which represent the Euler classes.

Using a cocycle in ${\Omega}^{*} (NG) $, we can construct a cocycle in the local truncated complex  $[\sigma_{<p}\Omega^{*}_{\rm loc}({NG}) ]$ due to
Brylinski \cite{Bry}. Furthermore, we can obtain a Lie algebra cocycle of a free loop group $LG$.
Following Brylinski's idea, we will construct a Lie algebra $2$-cocycle on $L\mathfrak{so}(4)$ using a cocycle in $\Omega^4(NSO(4))$.

\section{Review of the universal Chern-Weil Theory}

In this section we recall the universal Chern-Weil theory following \cite{Dup2}.
For any Lie group $G$, we have simplicial manifolds $N\bar{G}$, $N{G}$ and simplicial $G$-bundle  $\gamma : N \bar{G} \rightarrow NG$
as follows:\\
\par
$N \bar{G} (q) = \overbrace{ G \times \cdots \times G }^{q+1 - times} \ni (g_1 , \cdots , g_{q+1} ) $ \\

\par
\medskip
$NG(q)  = \overbrace{G \times \cdots \times G }^{q-times}  \ni (h_1 , \cdots , h_q ) :$  \\
face operators \enspace ${\varepsilon}_{i} : NG(q) \rightarrow NG(q-1)  $
$$
{\varepsilon}_{i}(h_1 , \cdots , h_q )=\begin{cases}
(h_2 , \cdots , h_q )  &  i=0 \\
(h_1 , \cdots ,h_i h_{i+1} , \cdots , h_q )  &  i=1 , \cdots , q-1 \\
(h_1 , \cdots , h_{q-1} )  &  i=q.
\end{cases}
$$

\par
\medskip

We define $\gamma : N \bar{G} \rightarrow NG $ as $ \gamma (g_0 , \cdots , g_{q} ) = (g_0 {g_1}^{-1} , \cdots , g_{q-1} {g_{q}}^{-1} )$.\\
\par
For any simplicial manifold $X = \{ X_* \}$, we can associate a topological space $\parallel X \parallel $ 
called the fat realization. It is well-known that $\parallel \gamma \parallel$ is the universal bundle $EG \rightarrow BG$  \cite{Seg}. 

Now we introduce a double complex associated to a simplicial manifold.

\begin{definition}
For any simplicial manifold $ \{ X_* \}$ with face operators $\{ {\varepsilon}_* \}$, we define a double complex as follows:
$${\Omega}^{p,q} (X) := {\Omega}^{q} (X_p) $$
Derivatives are:
$$ d' := \sum _{i=0} ^{p+1} (-1)^{i} {\varepsilon}_{i} ^{*}  , \qquad  d'' := (-1)^{p} \times {\rm the \enspace exterior \enspace differential \enspace on \enspace }{ \Omega ^*(X_p) }. $$
\end{definition}
\bigskip

For $NG$ and $N \bar{G} $ the following holds \cite{Bot2} \cite{Dup2} \cite{Mos}.

\begin{theorem}
There exist ring isomorphisms 
$$H({\Omega}^{*} (N \bar{G})) \cong H^{*} (EG ), \qquad   H({\Omega}^{*} (NG))  \cong  H^{*} (BG ). $$
Here ${\Omega}^{*} (N\bar{G})$  and  ${\Omega}^{*} (N{G})$  mean the total complexes.
\end{theorem}

\bigskip

There is another double complex associated to a simplicial manifold.
\begin{definition}[\cite{Dup}]
A simplicial $n$-form on a simplicial manifold $ \{ {X}_{p} \} $ is a sequence $ \{ {\phi}^{(p)} \}$
of $n$-forms ${\phi}^{(p)}$ on ${\Delta}^{p} \times {X}_{p} $ such that
$${({\varepsilon}^{i} \times id )}^{*} {\phi}^{(p)} = {(id \times {\varepsilon}_{i} )}^{*} {\phi}^{(p-1)}~~~{\rm on} ~~{\Delta}^{p-1} \times {X}_{p}.$$
Here ${\varepsilon}^{i}$ is the canonical $i$-th face operator of ${\Delta}^{p}$.\\
\end{definition}

{\rm Let}  \thinspace $A^{k,l} (X)$ be the set of all simplicial $(k+l)$-forms on ${\Delta}^{p} \times {X}_{p} $ which are expressed locally 
of the form
$$ \sum { a_{ i_1 \cdots i_k j_1 \cdots j_l } (dt_{i_1 } \wedge \cdots \wedge dt_{i_k } \wedge dx_{j_1 } \wedge \cdots \wedge dx_{j_l })}$$
where $(t_0, t_1, \cdots, t_p)$ are the barycentric coordinates in ${\Delta}^{p}  $ and $x_j $ are the local coordinates in $ {X}_{p} $.
We define derivatives as:
$$ d' := {\rm  the \enspace exterior \enspace differential \enspace on \enspace } {\Delta}^{p}  $$
$$ d'' := (-1)^{k} \times {\rm  the \enspace exterior \enspace differential \enspace on \enspace } {X_p }.$$
Then $(A^{k,l} (X) , d' , d'' )$ is a double complex and the following theorem holds.\\

\begin{theorem}[\cite{Dup}]
Let $A^{*} (X)$ denote the total complex of $A^{*,*}(X)$. A map $I_{\Delta} : A^{*} (X)  \rightarrow {\Omega}^{*} (X) $ defined as $ I_{ \Delta }( \alpha ) :=  \int_{{\Delta }^{p}} ( { \alpha } |_{{ \Delta }^{p} \times {X}_{p} } )$
 {induces a natural ring isomorphism} $ I_{ \Delta } ^{*} : H( A^{*} (X)) \cong H({\Omega}^{*} (X))$.

\end{theorem}
\bigskip

Let  $\mathcal{G}$ denote the Lie algebra of $G$. A connection on a simplicial $G$-bundle $\pi : \{ E_p \} \rightarrow \{ M_p \} $ is a sequence of $1$-forms $\{ \theta \}$ on $\{ E_p \}$ with coefficients $\mathcal{G}$
such that $\theta $ restricted  to ${\Delta}^{p} \times {E}_{p} $ is a usual connection form.

There is a canonical connection $\theta \in A^1 (N \bar{G} )$ on ${\gamma} : N \bar{G} \rightarrow NG $ defined as follows:
$$ {\theta } |_{{ \Delta }^{p} \times N \bar{G} (p)} := t_0 {\theta }_0 + \cdots + t_{p} {\theta }_{p}. $$

Here ${\theta }_i $ is defined as ${\theta }_i = {\rm pr}_i ^{*} \bar {\theta } $ where ${\rm pr}_i : { \Delta }^{p} \times N \bar{G} (p) \rightarrow G $ is the projection into the $(i+1)$-th factor of $ N \bar{G} (p) $ and $\bar {\theta }$ is the Maurer-Cartan form of $G$.
We obtain also its curvature $\Omega \in A^2 (N \bar{G} )$ on ${\gamma }$ as:
$$ \Omega |_{{ \Delta }^{p} \times N \bar{G} (p) }= d \theta |_{{ \Delta }^{p} \times N \bar{G} (p) }
+  \frac{1}{2} [ \theta |_{{ \Delta }^{p} \times N \bar{G} (p) } , \theta |_{{ \Delta }^{p} \times N \bar{G} (p) } ]. $$

Let  ${\rm  I}^{*} (G)$ denote the ring of $G$-invariant polynomials on $\mathcal{G}$. For $P \in I^k (G)$, we restrict 
$P( \Omega ) \in A^{2k} (N \bar{G} )$ to each ${\Delta}^{p} \times N \bar{G} (p) $ and apply the usual Chern-Weil theory then we have  $I_{\Delta} (P( \Omega )) \in \Omega^{2k}(NG)$. In this way we have a  homomorphism ${\rm  I}^{*} (G) \rightarrow H({\Omega}^{*} (NG)) $
which maps $P \in I^* (G)$  to $[I_{\Delta } ( P({\Omega}) )]$.

\section{The Euler class in the double complex}

In this section we exhibit a cocycle in $\Omega ^{*}(NSO(2p)) $  which represents the
Euler class of the universal bundle $ESO(2p) \rightarrow BSO(2p)$. Throughout this section,
$G$ means $SO(2p)$.

Recall that the polynomial on $\mathfrak{so}(2p)$ called Pfaffian is defined as follows:
$${\rm Pf}(A, \cdots, A)=\frac{1}{2^{2p}\pi ^{p}p!} \sum_{\tau \in \mathfrak{S} _{2p} } {\rm sgn}(\tau)a_{\tau (1)\tau (2)} \cdots a_{\tau (2p-1)\tau (2p)}.$$
Here $a_{ij}$ is a $(i,j)$ entry of $A \in \mathfrak{so}(2p)$.

\subsection{The cochain on the edge}

We first give the cochain in ${\Omega}^{2p+1}(N\bar{G}(1))$ which corresponds to the Euler class.
This is given by integrating ${\rm Pf} \left(  {  \Omega } |_{{ \Delta }^{1} \times N \bar{G} (1) }  \right) $ along ${\Delta}^{1}$. Since $  \Omega |_{{ \Delta }^{1} \times N \bar{G} (1) }
 = -  dt_1 \wedge (\theta _0 - \theta _{1} )
-  t_0 t_1 ( \theta _0 - \theta _1  ) ^{2} $, we can see ${\rm Pf} \left(  {  \Omega } |_{{ \Delta }^{1} \times N \bar{G} (1) }  \right)$ is equal to
$$ \frac{1}{2^{2p}\pi ^{p}p!} \sum_{\tau \in \mathfrak{S} _{2m}} {\rm sgn}(\tau)((-  dt_1 \wedge (\theta _0 - \theta _{1} )
-  t_0 t_1 ( \theta _0 - \theta _1  ) ^{2})_{\tau (1)\tau (2)}$$
$$~~~~~~~~~~~~~~~~~~~~~~~~ \cdots (-  dt_1 \wedge (\theta _0 - \theta _{1} )
-  t_0 t_1 ( \theta _0 - \theta _1  ) ^{2})_{\tau (2p-1)\tau (2p)}).$$
We set:
$$\bar{P}_{\tau} ^k:=( \theta _0 - \theta _1  ) ^{2} _{\tau (1)\tau (2)} \cdots ( \theta _0 - \theta _1  ) ^{2} _{\tau (2k-3)\tau (2k-2)}( \theta _0 - \theta _1  )_{\tau (2k-1)\tau (2k)}$$
$$~~~~~~~~~~~~~~~~~~~~~~~~~~~~~~~~~~~~~~~~~~( \theta _0 - \theta _1  ) ^{2} _{\tau (2k+1)\tau (2k+2)} \cdots ( \theta _0 - \theta _1  ) ^{2} _{\tau (2p-1)\tau (2p)}.$$
Then  the following equation holds.
$$\int_{{\Delta}^{1}} {\rm Pf} \left(  {  \Omega } |_{{ \Delta }^{1} \times N \bar{G} (1) }  \right)= (-1)^p  \frac{1}{2^{2p}\pi ^{p}p!}  \left( \int ^1 _0 (t_0 t_1)^{p-1} dt_1 \right)  \sum_{\tau \in \mathfrak{S} _{2p}} \sum_{k=1} ^p {\rm sgn} (\tau)  \bar{P}_{\tau} ^k . $$

Now we obtain the cochain in $ \Omega ^{2p-1} (NG(1)) $.

\begin{proposition}
The cochain $\mu_{p}$ in $ \Omega ^{2p-1} (NG(1)) $ which corresponds to the Euler class is given as follows:
$$\mu_{1}= (-1)^p \frac{1}{2^{2p}\pi ^{p}p!} \frac{1}{ {}_{2p-1}C_{p-1}\cdot p}  \sum_{\tau \in \mathfrak{S} _{2p}} \sum_{k=1} ^p  {\rm sgn} (\tau)  {P}_{\tau} ^k. $$

Here ${P}_{\tau} ^k$ is defined as:

$${P}_{\tau} ^k:=( h^{-1}dh ) ^{2} _{\tau (1)\tau (2)} \cdots ( h^{-1}dh   ) ^{2} _{\tau (2k-3)\tau (2k-2)}( h^{-1}dh  )_{\tau (2k-1)\tau (2k)}$$
$$~~~~~~~~~~~~~~~~~~~~~~~~~~~~~~~~~~~~(h^{-1}dh    ) ^{2} _{\tau (2k+1)\tau (2k+2)} \cdots ( h^{-1}dh  ) ^{2} _{\tau (2p-1)\tau (2p)}.$$

\end{proposition}
\bigskip

\begin{proof}
This follows from the equation $\displaystyle \int ^1 _0 (t_0 t_1)^{p-1} dt_1=\frac{1}{ {}_{2p-1}C_{p-1} \cdot p} $
and \\
$\displaystyle \gamma^*   \sum_{\tau \in \mathfrak{S} _{2p}}  {\rm sgn} (\tau) {P}_{\tau} ^k  =\sum_{\tau \in \mathfrak{S} _{2p}}  {\rm sgn} (\tau)\bar{P}_{\tau} ^k$.
\end{proof}

As a special case of Proposition 3.1, we obtain the following theorem.

\begin{theorem}
In the case of $G = SO(2)$, the cocycle $E_{1,1}$ in $ \Omega ^{2} (NG) $ which represents the Euler class of $ESO(2) \rightarrow BSO(2)$ is given as follows:  
$$E_{1,1}= \frac{1}{4 \pi}(-(h^{-1}dh)_{12}+(h^{-1}dh)_{21})~~\in \Omega^1(SO(2)). $$
If we write an element $h$ in $SO(2)$  as
\[ h = \left(
       \begin{array}{@{\,}cc@{\,}}
       \cos \theta & - \sin \theta \\
       \sin \theta & \cos \theta 
       \end{array}
       \right) \]
then the equation 
\[ h^{-1}dh = \left(
       \begin{array}{@{\,}cc@{\,}}
       0 & - d \theta \\
       d \theta & 0
       \end{array}
       \right) \]
holds, so we obtain
$$E_{1,1}= \frac{1}{4 \pi}(2 d \theta)=  \frac{d \theta}{2 \pi}. $$
\end{theorem}

\subsection{The cochain in $ \Omega ^{p} (NG(p)) $}

In $ \Omega ^{p} (N \bar{G}(p)) $, $\displaystyle  \Omega |_{{ \Delta }^{p} \times N \bar{G} (p) }$ is equal to
$\displaystyle -\sum _{i=1} ^{p}  dt_i \wedge (\theta _0 - \theta _{i} )-  \sum _{ 0 \leq i < j \leq p} t_i t_j ( \theta _i - \theta _j  ) ^{2}$,
so the cochain $\int_{{ \Delta }^{p}} {\rm Pf}( \Omega |_{{ \Delta }^{p} \times N \bar{G} (p) })$ in $ \Omega ^{p} (N\bar{G}(p)) $which corresponds to the Euler class is given as follows:
$$\frac{1}{2^{2p}\pi ^{p}p!} \sum_{\tau \in \mathfrak{S} _{2p}} {\rm sgn}(\tau) \bigl( -\sum _{i=1} ^{p}  dt_i \wedge (\theta _0 - \theta _{i} ) \bigl)_{\tau (1)\tau (2)}
\cdots  \bigl( -\sum _{i=1} ^{p}  dt_i \wedge (\theta _0 - \theta _{i} ) \bigl)_{\tau (2p-1)\tau (2p)}.$$

Now
$$ dt_i \wedge (\theta _0 -\theta _{i} ) = dt_i \wedge \{ (\theta _0 -\theta _{1} ) + (\theta _{1} -\theta _{2} ) +
\cdots + (\theta _{i-1} -\theta _{i} ) \} $$

and for any differential forms $ \alpha , \beta , \gamma $ and any integer $ 0 \leq k, l, x \leq p$, the equation
$  \alpha \wedge (dt_i \wedge ( \theta _x - \theta _{x+1} )_{\tau (2k-1)\tau (2k)}) \wedge \beta \wedge (dt_j \wedge ( \theta _x - \theta _{x+1} )_{\tau (2l-1)\tau (2l)}) \wedge \gamma  = - \alpha \wedge (dt_j \wedge ( \theta _x - \theta _{x+1} )_{\tau (2k-1)\tau (2k)}) \wedge \beta \wedge (dt_i \wedge ( \theta _x - \theta _{x+1} )_{\tau (2l-1)\tau (2l)}) \wedge \gamma $ holds,
so the terms of these forms cancel with each other in ${\rm Pf}( \Omega |_{{ \Delta }^{p} \times N \bar{G} (p) })$.

We set:
$$\varphi_s:=h_1 \cdots h_{s-1}dh_s h^{-1} _s \cdots h^{-1} _1.$$ 
Then we can check that $\gamma^*\varphi _s=g_1( \theta _{s-1} - \theta _{s} )g^{-1}_1$ hence we obtain the following proposition.

\begin{proposition}
The cochain $\mu_{p}$ in $ \Omega ^{p} (NG(p)) $ which corresponds to the Euler class is given as follows:
$$\mu_{p}= \frac{(-1)^{\frac{p(p+1)}{2}}}{2^{2p}\pi ^{p}(p!)^2}  \sum _{\sigma \in \mathfrak{S} _{p}} \sum _{\tau \in \mathfrak{S} _{2p}}  {\rm sgn} ( \tau ) {\rm sgn} ( \sigma ) (\varphi_{ \sigma (1)})_{\tau (1)\tau (2)} \cdots  (\varphi_{ \sigma (p) })_{\tau (2p-1)\tau (2p)}.$$
\end{proposition}
\bigskip

Using Proposition 3.1 and Proposition 3.2, we obtain the  cocycle which represents the Euler class of $ESO(4) \rightarrow BSO(4)$ in $ \Omega ^{4} (NSO(4)) $.

\begin{theorem}
In the case of $G=SO(4)$, the cocycle which represents the Euler class of $ESO(4) \rightarrow BSO(4)$ in $ \Omega ^{4} (NG) $ is the sum of the following
$E_{1,3}$ and $E_{2,2}$:
$$
\begin{CD}
0 \\
@AA{d''}A \\
E_{1,3} \in {\Omega}^{3} (SO(4) )@>{d'}>>{\Omega}^{3} (SO(4) \times SO(4))\\
@.@AA{d''}A\\
@.E_{2,2} \in {\Omega}^{2} (SO(4) \times SO(4))@>{d'}>> 0
\end{CD}
$$
$$E_{1,3} =  \frac{1}{192 \pi ^2} \sum_{\tau \in \mathfrak{S} _{4}}   {\rm sgn} (\tau)  \bigl((h^{-1}dh)_{\tau (1) \tau(2)}(h^{-1}dh)^2 _{\tau (3) \tau(4)}~~~~~~~~~~~~~~~~~~~~~~~~~~~~$$
$$~~~~~~~~~~~~~~~~~~~~~~~~~~~~~~~~~~~~~~~~~~~~~~~~~+(h^{-1}dh)^2 _{\tau (1) \tau(2)}(h^{-1}dh) _{\tau (3) \tau(4)} \bigl)$$
$$ E_{2,2} =  \frac{-1}{64 \pi ^2} \sum_{\tau \in \mathfrak{S} _{4}}   {\rm sgn} (\tau) \bigl((h_1 ^{-1}dh_1)_{\tau (1) \tau(2)}(dh_2 h_2^{-1}) _{\tau (3) \tau(4)}~~~~~~~~~~~~~~~~~~~~~~~~~~~~$$
$$~~~~~~~~~~~~~~~~~~~~~~~~~~~~~~~~~~~~~~~~~~~~~~+(dh_2h_2^{-1}) _{\tau (1) \tau(2)}(h_1^{-1}dh_1) _{\tau (3) \tau(4)} \bigl).$$

\end{theorem}

\subsection{The cocycle in $ \Omega ^{p+q} (NG(p-q)) $}

Repeating the same argument in section 3.2, we obtain a cocycle in $ \Omega ^{p+q} (NG(p-q)) $.

We set:
$$ R_{ij} := (\varphi_i + \varphi_{i+1}  + \cdots + \varphi_{j-1}  )^{2} \qquad (1 \leq i < j \leq p-q+1).$$

\begin{theorem} 
The cocycle in $ \Omega ^{p+q} (NG(p-q)) \ (0 \leq q \leq p-1)$  which represents the Euler class of $ESO(2p) \rightarrow BSO(2p)$
is
$$ \sum_{\sigma \in \mathfrak{S} _{p-q}, \tau \in \mathfrak{S} _{2p}}\sum (T^{\tau,\sigma}_{p,q} (R_{i_1j_1} )_{\tau (1)\tau (2)}( \varphi_{ \sigma (1)} )_{\tau (3)\tau (4)} ~~~~~~~~~~~~~~~~~~~~ $$
$$~~~~~~~~~~~~~~\cdots ( R_{i_qj_q})
_{\tau (2p-3)\tau (2p-2)}(\varphi_{ \sigma (p-q)})_{\tau (2p-1)\tau (2p)}) $$
where $ R_{ij} \enspace (1 \leq i < j \leq p-q+1 ) $  are put $q$-times between $\varphi_{ \sigma (l)}$ and  $\varphi_{ \sigma (l+1)}$ or the edge
 in $\varphi_{ \sigma (1)} \cdots  \varphi_{ \sigma (p-q)}$ permitting overlaps and  $\sum$ means the sum of all such forms. $T^{\tau,\sigma}_{p,q}$ is defined as:
$$T^{\tau,\sigma}_{p,q}={\rm sgn} ( \tau ) {\rm sgn} ( \sigma ) \frac{(-1)^{p+\frac{(p-q)(p-q-1)}{2}}}{2^{2p}\pi ^{p}p!} \left( \int _{\Delta ^{p-q}} \prod_{i<j} (t_{i-1} t_{j-1})^{r_{ij}}  dt_1 \wedge \cdots \wedge dt_{p-q} \right) $$
where  $r_{ij}$ means the number of $R_{ij} $ in each form.

\end{theorem}

\begin{theorem}
In the case of $G=SO(6)$, the cocycle which represents the Euler class in $ \Omega ^{6} (NG) $ is the sum of the following
$E_{1,5} , E_{2,4}$ and $E_{3,3}$:
$$
\begin{CD}
0 \\
@AA{d''}A \\
{ E_{1,5} \in {\Omega}^{5} (G )}@>{d'}>>{\Omega}^{5} (NG(2))\\
@.@AA{d''}A\\
@.{ E_{2,4} \in {\Omega}^{4} (NG(2))}@>{d'}>> {\Omega}^{4} (NG(3)) \\
@.@.@AA{d''}A\\
@.@.{ E_{3,3} \in {\Omega}^{3} (NG(3))}@>{d'}>> 0
\end{CD}
$$

$$E_{1,5} = \frac{-1}{2^6 \cdot 180 \pi ^3} \sum_{\tau \in \mathfrak{S} _{6}}   {\rm sgn} (\tau)   \bigl((h^{-1}dh)^2_{\tau (1) \tau(2)}(h^{-1}dh) _{\tau (3) \tau(4)} ({h^{-1}dh} )_{\tau (5) \tau(6)}~~~~~~~~~ $$
$$~~~~~~~~~~~~+(h^{-1}dh)_{\tau (1) \tau(2)}(h^{-1}dh)^2 _{\tau (3) \tau(4)} ({h^{-1}dh} )_{\tau (5) \tau(6)}$$
$$~~~~~~~~~~~~~~~~~~~~~~~~~+(h^{-1}dh)_{\tau (1) \tau(2)}(h^{-1}dh) _{\tau (3) \tau(4)} ({h^{-1}dh} )^2_{\tau (5) \tau(6)}\bigl)$$

$$    E_{2,4}= \frac{1}{2^6 \cdot 6 \cdot 4! \pi ^3} \sum_{\tau \in \mathfrak{S} _{6}}   {\rm sgn} (\tau) \cdot~~~~~~~~~~~~~~~~~~~~~~~~~~~~~~~~~~~~~~~~~~~~~~~~~~ $$

$$\biggl((h_1 ^{-1}dh_1)_{\tau (1) \tau(2)}(dh_2 h_2^{-1}) _{\tau (3) \tau(4)}\cdot
~~~~~~~~~~~~~~~~~~~~~~~~~~~~~~~~~~~~~~~~~~~~~~~~ $$
$$\Bigl( 2h_1^{-1}dh_1h_1^{-1}dh_1+ 2 dh_2h_2^{-1}dh_2h_2^{-1}+ h_1^{-1}dh_1dh_2h_2^{-1}+dh_2h_2^{-1}h_1^{-1}dh_1 \Bigl)_{\tau (5) \tau(6)} $$

$$+ (h_1 ^{-1}dh_1)_{\tau (1) \tau(2)}
\Bigl( 2h_1^{-1}dh_1h_1^{-1}dh_1+ 2 dh_2h_2^{-1}dh_2h_2^{-1} ~~~~~~~$$
$$+ h_1^{-1}dh_1dh_2h_2^{-1}+dh_2h_2^{-1}h_1^{-1}dh_1 \Bigl)_{\tau (3) \tau(4)}(dh_2 h_2^{-1}) _{\tau (5) \tau(6)}$$

$$+
\Bigl( 2h_1^{-1}dh_1h_1^{-1}dh_1+ 2 dh_2h_2^{-1}dh_2h_2^{-1} + h_1^{-1}dh_1dh_2h_2^{-1}+dh_2h_2^{-1}h_1^{-1}dh_1 \Bigl)_{\tau (1) \tau(2)}$$
$$~~~~~~~~~~~~~~~~~~~~~~~\cdot (h_1 ^{-1}dh_1)_{\tau (3) \tau(4)}(dh_2 h_2^{-1}) _{\tau (5) \tau(6)}$$

$$-(dh_2 h_2^{-1})_{\tau (1) \tau(2)} (h_1 ^{-1}dh_1) _{\tau (3) \tau(4)}\cdot
~~~~~~~~~~~~~~~~~~~~~~~~~~~~~~~~~~~~~~~~~~~~~~~ $$
$$\Bigl( 2h_1^{-1}dh_1h_1^{-1}dh_1+ 2 dh_2h_2^{-1}dh_2h_2^{-1}+ h_1^{-1}dh_1dh_2h_2^{-1}+dh_2h_2^{-1}h_1^{-1}dh_1 \Bigl)_{\tau (5) \tau(6)}$$

$$-(dh_2 h_2^{-1})_{\tau (1) \tau(2)} 
\Bigl( 2h_1^{-1}dh_1h_1^{-1}dh_1+ 2 dh_2h_2^{-1}dh_2h_2^{-1} ~~~~~~~~$$
$$+ h_1^{-1}dh_1dh_2h_2^{-1}+dh_2h_2^{-1}h_1^{-1}dh_1 \Bigl)_{\tau (3) \tau(4)}(h_1 ^{-1}dh_1) _{\tau (5) \tau(6)}$$

$$-
\Bigl( 2h_1^{-1}dh_1h_1^{-1}dh_1+ 2 dh_2h_2^{-1}dh_2h_2^{-1}+ h_1^{-1}dh_1dh_2h_2^{-1}+dh_2h_2^{-1}h_1^{-1}dh_1 \Bigl)_{\tau (1) \tau(2)} $$
$$~~~~~~~~~~~~~~~~~~~~~~~~~~~~~~~\cdot (dh_2 h_2^{-1})_{\tau (3) \tau(4)} (h_1 ^{-1}dh_1) _{\tau (5) \tau(6)}\biggl).$$

$$  E_{3,3} = \frac{1}{2^6 \cdot 6^2 \pi ^3} \sum_{\tau \in \mathfrak{S} _{6}}   {\rm sgn}(\tau) \cdot~~~~~~~~~~~~~~~~~~~~~~~~~~~~~~~~~~~~~~~~~~~~~~~$$
$$\biggl((h_1 ^{-1}dh_1)_{\tau (1) \tau(2)}(dh_2 h_2^{-1}) _{\tau (3) \tau(4)}
(h_2dh_3h_3^{-1}h_2^{-1})_{\tau (5) \tau(6)}~~~~~~~~~~~ $$
$$-(dh_2 h_2^{-1})_{\tau (1) \tau(2)}(h_1 ^{-1}dh_1) _{\tau (3) \tau(4)}
(h_2dh_3h_3^{-1}h_2^{-1})_{\tau (5) \tau(6)}~~~~~~~$$
$$-(h_1 ^{-1}dh_1)_{\tau (1) \tau(2)}(h_2dh_3h_3^{-1}h_2^{-1}) _{\tau (3) \tau(4)}
(dh_2 h_2^{-1})_{\tau (5) \tau(6)}~~~  $$
$$ +(h_2dh_3h_3^{-1}h_2^{-1})_{\tau (1) \tau(2)}(h_1 ^{-1}dh_1) _{\tau (3) \tau(4)}
(dh_2 h_2^{-1})_{\tau (5) \tau(6)} $$
$$~~~+(dh_2 h_2^{-1})_{\tau (1) \tau(2)}(h_2dh_3h_3^{-1}h_2^{-1}) _{\tau (3) \tau(4)}
(h_1 ^{-1}dh_1)_{\tau (5) \tau(6)}  $$
$$~~~~~~~-(h_2dh_3h_3^{-1}h_2^{-1})_{\tau (1) \tau(2)}(dh_2 h_2^{-1}) _{\tau (3) \tau(4)}
(h_1 ^{-1}dh_1)_{\tau (5) \tau(6)}\biggl).$$

\end{theorem}

\section{The cocycle in a local truncated complex}

We recall the filtered local simplicial de Rham complex due to Brylinski \cite{Bry}.

\begin{definition}[\cite{Bry}]
The filtered local simplicial de Rham complex $F^p\Omega_{{\rm loc}} ^{*,*}(NG)$ over a simplicial manifold $NG$ is
defined as follows:\\
$$
F^p \Omega^{r,s} _{\rm loc}(NG)=\begin{cases}
 \underrightarrow{\rm lim}_{1 \in V \subset G^r} ~~\Omega ^s(V)  & ~{\rm if}~ s \ge p \\
0  &  {\rm otherwise}. 
\end{cases}
$$
\end{definition}

Let $F^p \Omega^{*}(NG)$ be a filtered complex 
$$
F^p \Omega^{r,s} (NG)=\begin{cases}
 \Omega ^s(NG(r))  & ~{\rm if}~ s \ge p \\
0  &  {\rm otherwise} 
\end{cases}
$$
and $[\sigma_{<p}\Omega ^{*}({NG})]$ a truncated complex
$$
[\sigma_{<p}\Omega ^{r,s}({NG})]=\begin{cases}
0  &  ~{\rm if}~ s \ge p  \\
\Omega ^s(NG(r))  & ~{\rm otherwise}. 
\end{cases}
$$
Then there is an exact sequence: 
$$0 \rightarrow F^p \Omega^{*}(NG) \rightarrow \Omega^{*}(NG) \rightarrow [\sigma_{<p}\Omega^{*}({NG}) ] \rightarrow 0$$
which induces a boundary map $\beta:H^l(NG,[\sigma_{<p}\Omega_{\rm loc} ^{*}]) \rightarrow H^{l+1}(NG,[F^p \Omega^{*} _{\rm loc}]).$

Let $\mu_1 + \cdots + \mu_p$, $\mu_{p-q} \in \Omega^{p+q}(NG(p-q))$
be a cocycle in $\Omega^{2p}(NG)$. Using this cocycle, we can construct a cocycle $\eta$ in $[\sigma_{<p}\Omega ^{*} _{\rm loc}({NG})]$ in the following way.

We take a contractible open set $U \subset G$ containing $1$. Using the same argument in \cite{Dup2}, we can 
construct mappings $\{ \sigma_l:{ \Delta }^l\times  U^l \rightarrow U \}_{0 \le l}$ inductively with the following properties:\\
(1)~$\sigma_0(pt)=1$;\\
(2)
$$
{\sigma}_l({\varepsilon}^{j}(t_0, \cdots , t_{l-1});h_1, \cdots, h_l)=\begin{cases}
{\sigma}_{l-1}( t_0, \cdots, t_{l-1};{\varepsilon}_{j}(h_1, \cdots , h_l))   &  ~{\rm if}~j \ge 1 \\
h_1 \cdot {\sigma}_{l-1}(t_0, \cdots, t_{l-1};h_2, \cdots , h_l)   &  ~{\rm if}~j=0.
\end{cases}
$$
We define mappings $\{f_{m,q}: { \Delta }^q \times  U^{m+q-1} \rightarrow G^m \}$ as
$$f_{m,q}(t_0, \cdots , t_q;h_1, \cdots, h_{m+q-1}):=(h_1, \cdots , h_{m-1}, {\sigma}_q(t_0, \cdots , t_q;h_m, \cdots , h_{m+q-1})).$$
 A $(2p-m-q)$-form $\beta_{m,q}$ on $U^{m+q-1}$ is defined as $\beta_{m,q}=(-1)^m \int_{{ \Delta }^q}f_{m,q} ^* {\mu}_m$.
Then we define the cochain $\eta$ as the sum of following $\eta_l$ on $U^{2p-1-l}$ for $0 \le l \le p-1$:
$$\eta _l:=\sum_{m+q=2p-l,~ p \ge m \ge 1}\beta _{m,q}.$$

\begin{theorem}[\cite{Bry}\cite{Suz}]
$\eta := \eta _0+ \cdots + \eta_{p-1}$ is a cocycle in $[\sigma_{<p}\Omega_{\rm loc} ^{*}(NG)]$ whose cohomology class is mapped 
to $[\mu_1 + \cdots + \mu_p]$ in $H^{2p}(NG,[F^p \Omega^{*} _{\rm loc}])$ by a boundary map 
$\beta : H^{2p-1}(NG,[\sigma_{<p}\Omega_{\rm loc} ^{*}]) \rightarrow H^{2p}(NG,[F^p \Omega^{*} _{\rm loc}])$.
\end{theorem}
\begin{proof}
See \cite{Suz}.
\end{proof}

\section{Construction of a Lie algebra cocycle}
For any Lie group $G$, let $C^{\infty} _{loc}(G^p, {\mathbb R})$ denote the group of germs at $(1, \cdots, 1)$ of smooth functions $G^p \rightarrow {\mathbb R}$ and  $H_{loc} ^p(G, {\mathbb R})$ denote the cohomology group of the following complex:
$$ \cdots \rightarrow C^{\infty} _{loc}(G^p, {\mathbb R}) \xrightarrow{\delta := \sum _{i=0} ^{p+1} (-1)^{i} {\varepsilon}_{i} ^{*}}  C^{\infty} _{loc}(G^{p+1}, {\mathbb R}) \rightarrow \cdots $$

Brylinski constructed a natural cochain map $\phi:C^p_{loc}(G, {\mathbb R}) \rightarrow C^p(\mathcal{G}, {\mathbb R})$ as follows:
$$\phi(c)( \xi_1, \cdots , \xi_p):=~~~~~~~~~~~~~~~~~~~~~~~~~~~~~~~~~~~~~~~~~~~~~~~~~~~~~~~~~~~~~~~$$
$$[\frac{ \partial ^p}{\partial y_1 \cdots \partial y_p} \sum_{\rho \in \mathfrak{S} _{p}} {\rm sgn} (\rho)c( \exp (y_{\rho(1)} \xi_{\rho(1)}), \cdots ,\exp (y_{\rho(p)} \xi_{\rho(p)}))]_{y_i=0}$$
where $C^p(\mathcal{G}, {\mathbb R})$ is the space of smooth alternating multilinear maps $\mathcal{G} \rightarrow {\mathbb R}$
and $\xi_i \in \mathcal{G}$.
For example, if we take $\delta c \in C^{\infty} _{loc}(G^2, {\mathbb R})$ and set $X_{\rho(i)}:=\exp (y_{\rho(i)} \xi_{\rho(i)})$ then 
$$\phi(\delta c)( \xi_1, \xi_2 )=[\frac{ \partial ^2}{\partial y_1  \partial y_2} \sum_{\rho \in \mathfrak{S} _{2}} {\rm sgn} (\rho) (\delta c)( X_{\rho(1)}, X_{\rho(2)})]_{y_i=0}$$
$$=[\frac{ \partial ^2}{\partial y_1  \partial y_2} \sum_{\rho \in \mathfrak{S} _{2}} {\rm sgn} (\rho)
  (c(X_{\rho(2)})-c( X_{\rho(1)} X_{\rho(2)} ) + c(X_{\rho(1)})]_{y_i=0})$$
$$=[\frac{ \partial ^2}{\partial y_1  \partial y_2}( - c( X_1 X_2 -X_2X_1)) ]_{y_i=0}=(d(\phi(c)))(\xi_1, \xi_2).$$

Let $LU$ be the free loop space of a contractible open set $U \subset SO(4)$ containing $1$ and ${\rm ev}:LU \times  S^1 \to U$ be the evaluation map, i.e. for $\gamma \in LU$ and $\theta \in S^1$, ${\rm ev}(\gamma , \theta)$ is defined as $\gamma (\theta)$. Then $\int _{S^1} {\rm ev}^*$ maps $\eta_{1} \in \Omega ^1(U^{2})$ to a cochain in  $\Omega ^0(LU^{2})$.
This cochain defines a cohomology class in local cohomology group $H^{2} _{\rm loc}(LSO(4), {\mathbb R})$. 
So as an application of Theorem 3.2, we can obtain a cocycle in  $\phi  (\int _{S^1} {\rm ev}^* \eta_1) \in C^2(L\mathfrak{so}(4), {\mathbb R})$.

Now we compute this cocycle. We define:
$$a:=\int _{S^1} {\rm ev}^*\int_{\Delta^2}f_{1,2}^*E_{1,3},~~~b:=\int _{S^1} {\rm ev}^*\int_{\Delta^1}f_{2,1}^*E_{2,2},~~~c:=\int _{S^1} {\rm ev}^* \eta_1$$
then $c(\gamma_1,\gamma_2)=a(\gamma_1,\gamma_2)+b(\gamma_1,\gamma_2)$ for $\gamma_1, \gamma_2 \in LU$.
Recall that 
$$ f_{1,2}(t_0, t_1, t_2;\gamma_1(\theta), \gamma_2(\theta))=\sigma_2(t_0, t_1, t_2;\gamma_1(\theta), \gamma_2(\theta))$$
$$f_{2,1}(t_0, t_1, t_2;\gamma_1(\theta), \gamma_2(\theta))=(\gamma_1(\theta),\sigma_1( t_0, t_1;\gamma_2(\theta))).$$
In this case we can take:
$$\gamma_i(\theta)=\exp(y_i \xi_i(\theta))$$
$$\sigma_1 (t_0,t_1; \exp(y_2\xi_2(\theta))):=\exp(t_1y_2\xi_2(\theta))$$
$$\sigma_2 (t_0,t_1,t_2; \exp (y_1\xi_1(\theta)),\exp (y_2\xi_2(\theta))):=\exp ((1-t_0)y_1\xi_1(\theta))\exp (t_2y_2\xi_2(\theta))$$
where $\xi_i \in L\mathfrak{so}(4)$.
By observing the coefficient of $y_1y_2$, we see  $\phi(a(\gamma_1,\gamma_2))=0$.

We define a map  $\beta_{\gamma_1, \gamma_2}:S^1 \times \Delta^1 \rightarrow SO(4) \times SO(4)$ as follows:
$$\beta_{\gamma_1, \gamma_2}(\theta; t_0, t_1) :=(\gamma_1(\theta),\sigma_1( t_0, t_1;\gamma_2(\theta))).$$

Then $b(\gamma_1,\gamma_2)=\int_{S^1 \times \Delta^1} \beta_{\gamma_1, \gamma_2} ^*E_{2,2}$ and up to $O(|y_1|^2)$ and $O(|y_2|^2)$,
$$\frac{\partial \beta_{\gamma_1, \gamma_2}}{\partial \theta}=\left(y_1\frac{\partial\xi_1(\theta)}{\partial\theta},t_1y_2\frac{\partial\xi_2(\theta)}{\partial\theta} \right),
~~~~\frac{\partial \beta_{\gamma_1, \gamma_2}}{\partial t_1}=\left(0, y_2\xi_2(\theta)\right).$$
Therefore
$$[\frac{ \partial ^2}{\partial y_1  \partial y_2}b(\gamma_1,\gamma_2)]_{y_i=0}=  \frac{-1}{128 \pi ^2} \sum_{\tau \in \mathfrak{S} _{4}}   {\rm sgn} (\tau) \int^1_0 \left( \frac{\partial\xi_1(\theta)}{\partial\theta} \right)_{\tau(1)\tau(2)} \xi_2(\theta)_{\tau(3)\tau(4)}d\theta.$$
Now we obtain the following theorem.
\begin{theorem}
There exists a Lie algebra $2$-cocycle $\alpha$ on $L\mathfrak{so}(4)$ which is expressed as follows:
$$\alpha(\xi_1,\xi_2):=\frac{-1}{128 \pi ^2} \sum_{\tau \in \mathfrak{S} _{4}} \biggl( {\rm sgn} (\tau) \cdot ~~~~~~~~~~~~~~~~~~~~~~~~~~~~~~~~~~~~~~~~~~~~~ $$
$$ \int^1_0 \Bigl( \left( \frac{\partial\xi_1(\theta)}{\partial\theta} \right)_{\tau(1)\tau(2)} \xi_2(\theta)_{\tau(3)\tau(4)} -\left( \frac{\partial\xi_2(\theta)}{\partial\theta} \right)_{\tau(1)\tau(2)} \xi_1(\theta)_{\tau(3)\tau(4)} \Bigl) d\theta \biggl).$$
\end{theorem}

National Institute of the Technology, Akita College, 1-1, Iijima Bunkyo-cho, Akita-shi, Akita-ken, Japan. \\
e-mail: nysuzuki@akita-nct.ac.jp
\end{document}